\documentclass[11pt]{amsart}
\usepackage{amssymb, url,mathrsfs}

\title{Multiples of integral points on elliptic curves}
\author{Patrick Ingram}
\address{Department of Mathematics, University of Waterloo}
\email{pingram@math.uwaterloo.ca}
\date{revised August 2008}
\thanks{This research was supported in part by a postdoctoral fellowship from NSERC of Canada.}

\newcommand{\QQ}{\mathbb{Q}}
\newcommand{\ZZ}{\mathbb{Z}}
\newcommand{\CC}{\mathbb{C}}
\newcommand{\RR}{\mathbb{R}}

\newcommand{\Ocal}{\mathcal{O}}
\newcommand{\h}{\hat{h}}
\renewcommand{\Re}{\mathrm{Re}}
\renewcommand{\Im}{\mathrm{Im}}

\newcommand{\ord}{\operatorname{ord}}
\newcommand{\Res}{\operatorname{Res}}
\newcommand{\MOD}[1]{~(\textup{mod}~#1)}  
\newcommand{\Lang}{C_\lambda}

\newtheorem{theorem}{Theorem}
\newtheorem{proposition}[theorem]{Proposition}
\newtheorem{claim}[theorem]{Claim}
\newtheorem{lemma}[theorem]{Lemma}

\newtheorem*{corollary}{Corollary}

\theoremstyle{definition}
\newtheorem*{remark}{Remark}

\begin{document}

\begin{abstract}
If $E$ is a minimal elliptic curve defined over $\ZZ$, we obtain a bound $C$, depending only on the global Tamagawa number of $E$,  such that for any point $P\in E(\QQ)$, $nP$ is integral for at most one value of $n>C$.  As a corollary, we show that if $E/\QQ$ is a fixed elliptic curve, then for all twists $E'$ of $E$ of sufficient height, and all torsion-free, rank-one subgroups $\Gamma\subseteq E'(\QQ)$, $\Gamma$ contains at most $6$ integral points.  Explicit computations for congruent number curves are included.
\end{abstract}

\maketitle

\section*{Introduction}

When considering the subject of integral points on elliptic curves, it seems natural to ask which multiples of a non-torsion point may be integral.  
 Let $E/\QQ$ be an elliptic curve and $P\in E(\QQ)$ be a point of infinite order. If $P$ is not integral, then one can easily show that it has no integral multiples.  A theorem of Siegel shows that $E(\QQ)$ contains at most finitely many integral points, and so it is certainly true that $P$ has at most finitely many integral multiples.  While it is possible to construct points on elliptic curves with arbitrarily many integral multiples, these constructions are somewhat artificial, and may be avoided if one considers only minimal curves.  With this restriction, it seems likely that the number of integral multiples of $P$ is bounded uniformly.  Certainly, if one assumes the $abc$ Conjecture of Masser and Oesterl\'{e}, then it follows from work of  Hindry and Silverman \cite{hindrysilv} that a uniform bound exists.  If one restricts attention to  curves with integral $j$-invariant --- or more generally, curves with at most a fixed number of primes dividing the denominator of the $j$-invariant ---  work of Silverman \cite{silvsieg} provides the same result unconditionally. 

With respect to bounding the \emph{size} of the largest $n$ such that $nP$ is integral, much less appears to be known.  For a given point on a given curve, the techniques arising from the study of linear forms in elliptic logarithms give an effective method for bounding $n$, but the bound depends strongly on the curve and, indeed, the point $P$.  
  By a careful consideration of division polynomials of elliptic curves, we are able to make the dependence on $P$ and $E$ more explicit.  Note that the idea of using the sequence of division polynomials on $E$ to say something about the integrality of multiples of a point is not new; specifically, the reader is directed to the work of Ayad \cite{ayad}.  In the present case, however, we are able to   
  bound $n$ such that $nP$ is integral in terms of the height of $E$, and another quantity, $M(P)$, related to the Tamagawa number of $E$.  For any elliptic curve $E/\QQ$, and each prime $p$, the connected component $E_0(\QQ_p)\subseteq E(\QQ_p)$ is a subgroup of finite index \cite[p.~385]{silverman}, this index being 1 at all primes of good reduction.  For $P\in E(\QQ)$, let $r(P, p)$ denote the order of $P$ in the quotient group $E(\QQ_p)/E_0(\QQ_p)$.  We will set
  $$M(P)=\operatorname{lcm}\{r(P, p)\},$$
as $p$ varies over all primes.  When $P$ is fixed, we will simply refer to $M$.

 Although 
the bound on the largest $n$ such that $nP$ is integral depends on the height of the curve, it affords us, perhaps surprisingly, a bound on the \emph{second} largest such value which depends only on $M$.  In an argument not dissimilar to that behind the proof of Thue's Theorem on diophantine approximation, we assume the existence of a very large $n$ such that $nP$ is integral, and then bound all other such $n$.  The bound obtained for `all but one' of the positive integers $n$ such that $nP$ is integral is independent of the point $P$ and the curve $E$, and can be presented entirely explicitly in terms of the quantity $M$.

\begin{theorem}\label{th:main}
There is an absolute constant $C$ such that for all minimal elliptic curves $E/\QQ$, and non-torsion points $P\in E(\QQ)$, there is at most one value of $n>CM(P)^{16}$ such that $nP$ is integral.  Furthermore, this one value is prime.
\end{theorem}

Note that if one restricts attention to elliptic curves with $j(E)\in\mathbb{Z}$, for instance, one always has $M(P)\leq 12$, and so the bound in the theorem is absolute.  Similarly, all curves in a family of quadratic twists will have the same $j$-invariant, and so there is a $C'=C'(j)$ such that $nP$ is integral for at most one value of $n>C'$ (independent of $M(P)$).  In fact, one can do much better. 
 Applying work of the author and Silverman \cite{meandjoe} one may, for each $n\leq C'(j)$, effectively find all examples of points $P$ on twists in our family for which $nP$ is integral.
  Thus one may take the constant $C'$ to be  2, or in some cases 1, modulo a finite, effectively computable set of exceptions.
  We obtain a particularly explicit result for the family of congruent number curves.

\begin{theorem}\label{th:cong}
Let $N$ be a square-free integer,  let $$E_N:y^2=x^3-N^2x,$$ and let $P\in E_N(\QQ)$ be a non-torsion integral point.  Then there is at most one value of $n>1$ such that $nP$ is integral.
\end{theorem}

  It is the size of this bound, not its existence, that is novel.  Gross and Silverman \cite{grosssilv} derived an explicit version of the result of Silverman \cite{silvsieg} mentioned above, which bounds, as a special case, the number of integral points on any rank-one elliptic curve  $E/\QQ$ by $3.3\times 10^{33}$, provided $j(E)\in\ZZ$.  It is worth noting, as well, that in light of Lemma~\ref{lem:timestwo} below, Theorem~1 of \cite{me:eds} implies that $P\in E_N(\QQ)$ has no integral multiples (other than $\pm P$) if $x_P<0$, or if $x_P$ is a square. 

Throughout, we will assume that $E:y^2=x^3+Ax+B$ is an elliptic curve in short Weierstrass form, with integral coefficients.  For such a curve, we define the (logarithmic) height to be
$$h(E)=\max\{h(j(E)), \log\max\{4|A|, 4|B|\}\}\geq 2\log 2,$$
where $j(E)=1728(4A^3)/(4A^3+27B^2)$ is the usual $j$-invariant and $h(p/q)=\log\max\{|p|, |q|\}$ is the  usual logarithmic height on $\QQ$.
  We define the canonical height of a point $P\in E(\QQ)$ to be
$$\h(P)=\frac{1}{2}\lim_{n\rightarrow\infty}\frac{h(x_{2^nP})}{4^n}$$
as in \cite{silverman}.  This differs from the height used in \cite{david} by a factor of 2, and while this is immaterial for the general result, we mention this for the benefit of
the reader wishing to recreate the explicit calculations in the later sections.
We will say that $E$ is quasi-minimal if $\Delta(E)$ is minimal within the $\QQ$-isomorphism class of $E$, subject to the constraint that $E$ have the form above.  Such curves may not be minimal, in the usual sense, at 2 or~3, but the extent of their non-minimality is bounded.

The paper proceeds as follows: In Section~\ref{sec:eds} we show that there is a uniform constant $C$ such that if $E/\QQ$ is quasi-minimal and $nP$ is integral, for some $P\in E(\QQ)$, then $n\leq CM^{16}$ or $n$ is prime.  
We show, in Section~\ref{sec:linforms}, that any integer $n$ such that $nP$ is integral  satisfies $n\ll h(E)^{5/2}$, where the implied constant depends only on $M$.   In Section~\ref{sec:gap} we complete the proof of Theorem~\ref{th:main} by constructing a function $f(x,y)$ such that if $n_1P$ and $n_2P$ are both integral, and $n_1, n_2$ are `large', then $f(n_1, n_2)$ is `small'.  Using the result of Section~\ref{sec:eds} as a non-vanishing result, we construct an elementary lower bound on $f(n_1, n_2)$, and a contradiction ensues.
  Note that as there are only finitely many curves below any given height, we may  (effectively) find all integral points on curves with height below the bound, and thereby check that the theorem holds.    
    We do not take great pains to keep explicit track of the constants that arise, only the dependence on $M$.  One obtains uniform bounds in contexts where $M$ is uniformly bounded, but in these special cases it is best to optimize the entire proof for the setting at hand, as in Section~\ref{sec:examples}, where we prove Theorem~\ref{th:cong}.

Although we consider only the problem of integral points on curves, it turns out that these methods allow one to to prove (weaker) results about points satisfying weaker diophantine constraints.  This is explored further in~\cite{me_quant}.


\section{Elliptic divisibility sequences and division polynomials}\label{sec:eds}

Throughout this paper, we make use of ideas from the study of elliptic divisibility sequences.  If $E/\QQ$ is an elliptic curve, and $P\in E(\QQ)$ is any non-torsion point, we may write
$$x_{nP}=\frac{A_n}{D_n^2}$$
in lowest terms taking, without loss of generality, $D_n>0$.  The problem of finding all $k$ such that $kP$ is integral is, of course, the same as describing all $k$  such that $D_k=1$.  This is a weaker version of the problem of determining which terms in the sequence $(D_n)_{n\in\ZZ}$ fail to have primitive divisors (i.e., prime divisors not dividing earlier terms in the sequence), and so we may apply the results of \cite{me:eds, meandjoe} to the problem at hand.  Although we are interested in the sequence $(D_n)_{n\in\ZZ}$, it benefits us to consider a related divisibility sequence.

Ward \cite{ward} examines sequences of integers $(h_n)_{n\in\ZZ}$ such that $h_0=0$, $h_1=1$, $h_2h_3\neq 0$, $h_2\mid h_4$, and such that
\begin{equation}\label{ellipticprop}
h_{m+n}h_{m-n}=h_{m-1}h_{m+1}h_n^2-h_{n-1}h_{n+1}h_m^2
\end{equation}
for all indices $m$ and $n$.
It is not entirely obvious that such a sequence is a divisibility sequence, that any initial values $h_2$, $h_3$, and $h_4$ define such a sequence, or that they do so uniquely, but the proofs of these three claims may be found in \cite{ward}.  

Supposing $E:y^2=x^3+Ax+B$ is an elliptic curve over $\QQ$, and $P=(a,b)\in E(\QQ)$ is a non-torsion point with integral coordinates, we may associate to $E$ and $P$ a Ward-type divisibility sequence $(h_n)_{n\in\ZZ}$ by setting $h_0=0$, $h_1=1$, $h_2=2b$, 
\begin{gather*}
h_3=3a^4+6Aa^2+12Ba-A^2,\\
\intertext{ and }
h_4=4b(a^6+5Aa^4+20Ba^3-5A^2a^2-4ABa-8B^2-A^3).
\end{gather*}
This matches precisely the definition of the division polynomials of $E$, and indeed one can verify that
$$h_n=\psi_n(P),$$
with $\psi_n$ defined as in \cite[p.~105]{silverman}, a fact that we will employ below.  We similarly define an auxiliary sequence $k_n=\phi_n(P)$, where
\begin{equation}\label{auxiliary}\phi_n=x\psi_n^2-\psi_{n+1}\psi_{n-1},\end{equation}
as in  \cite[p.~105]{silverman}.
We  then have
$$x_{nP}=\frac{k_n}{h_n^2}.$$
We do not, in general, have $D_{n}=|h_n|$, as it is perhaps true that $(h_n, k_n)\neq 1$, but we may control the extent of the cancellation in this fraction.

\begin{lemma}\label{twoeds}
Let $E/\QQ$ be an elliptic curve, let $P\in E(\QQ)$ be a point of infinite order, and let $h_n$, $D_n$, and $M$  be as  defined above.  Then for $n\geq 1$,
$$\log D_n\leq \log |h_n| \leq \log D_n+n^2M^2\log|\Delta(E)|.$$
\end{lemma}

\begin{proof}
Since we have $$\frac{A_n}{D_n^2}=x_{nP}=\frac{k_n}{h_n^2},$$
with $A_n, D_n, k_n, h_n\in \ZZ$, and $\gcd(A_n, D_n)=1$, the first inequality is immediate.  The second inequality amounts to bounding $\gcd(k_n, h_n^2)$.

 We will fix a prime $p$, and consider the order of $p$ in $h_n/D_n$.
Let $\phi_n$, $\psi_n$ be the division polynomials, and set $$g_n=\gcd(k_n, h_n^2)=\gcd(\phi_n(P), \psi^2_n(P)),$$
so that $$\ord_p(h_n) = \ord_p(D_n)+\frac{1}{2}\ord_p( g_n).$$
The quantity $g_n$ divides the resultant $\Res(\phi_n, \psi_n^2)=\Delta(E)^{\frac{1}{6}n^2(n^2-1)}$ (see equation~1.3 of \cite{ayad}), and so we have
\begin{equation}\label{resbound}\ord_p(g_n)\leq \frac{1}{6}n^2(n^2-1)\ord_p(\Delta(E))\end{equation}
for all $n$.

Now let $r\mid M$ be the order of $P$ in the component group $E(\QQ_p)/E_0(\QQ_p)$.
We invoke a result of Cheon-Hahn \cite{MR1654780}, which states that
$$\ord_p(g_n)=\begin{cases}
\ord_p(g_r) m^2 & \text{if }n=mr\\
4\ord_p(g_r) m^2\pm 2\left(2\ord_p\left(\frac{h_k}{h_{r-k}}\right)+\ord_p(g_r)\right)m\\\quad +2\ord_p(h_k) & \text{if }n=2mr\pm k,
\end{cases}$$
where $1\leq k<r$.

If $n=mr$, then we have
\begin{eqnarray*}
\ord_p(g_n)&=&\ord_p(g_r)\left(\frac{n}{r}\right)^2\\
& \leq &\frac{1}{6}(r^2-1)n^2\ord_p(\Delta(E))\\ &\leq& \frac{1}{6}n^2M^2\ord_p(\Delta(E)).\end{eqnarray*}

Now suppose that $n=2mr+k$, with $1<k<r$.  If $m=0$, then \eqref{resbound} gives
$$\ord_p(g_n)\leq \frac{1}{6}n^2(n^2-1)\ord_p(\Delta(E))\leq \frac{1}{6}n^2M^2\ord_p(\Delta(E)),$$
as $$n^2-1\leq n^2=k^2<r^2\leq M^2.$$  We will suppose, then, that $m\geq 1$, and so $n> 2mr \geq 2r$.
We have
$$4\ord_p(g_r)m^2\leq \ord_p(g_r)\left(\frac{n}{r}\right)^2\leq \frac{1}{6}n^2M^2\ord_p(\Delta(E)),$$
as above.  Now, note that $x_{kP}$ is singular, and so $\ord_p(x_{kP})\geq 0$.  In particular, $2\ord_p(h_k)=\ord_p(g_k)$, and
$$2\ord_p(h_k)\leq \frac{1}{6}k^2(k^2-1)\ord_p(\Delta(E))\leq \frac{1}{24}n^2M^2\ord_p(\Delta(E)),$$
as $2k<2r\leq n$.
Finally,
\begin{eqnarray*}
2\left(2\ord_p\left(\frac{h_k}{h_{r-k}}\right)+\ord_p(g_r)\right)m&\leq&2(\ord_p(g_k)+\ord_p(g_r))\frac{n}{2r}\\
&\leq &n^2\left(\frac{k^2(k^2-1)}{6rn}+\frac{r(r^2-1)}{6n}\right)\ord_p(\Delta(E))\\
&\leq&\frac{1}{6}n^2M^2\ord_p(\Delta(E)),
\end{eqnarray*}
again as $k<r\leq M$.  It follows that
$$\ord_p(g_n)\leq \frac{3}{8}M^2n^2\ord_p(\Delta(E)).$$  The case $n=2mr-k$, in which we obtain the bound
$$\ord_p(g_n)\leq 2n^2M^2\ord_p(\Delta(E)),$$
 is left to the reader. 

Thus, summing over all primes $p\mid\Delta(E)$, we have established that
$$\log g_n\leq 2n^2M^2\log|\Delta(E)|,$$ proving the result.
\end{proof}

  Using the relation
\begin{equation}\label{eq:divpoly}h_n^2=n^2\prod_{Q\in E[n]\setminus\{\Ocal\}}\left|x_P-x_Q\right|\end{equation}
we will produce a lower bound on $|h_n|$ given that $|x_P|$ is sufficiently large, allowing us to obtain a bound on the height of $P$.  This will be useful both in proving our assertion that large values of $n$ with $nP$ integral must be prime, as well as in obtaining bounds on $n$ in Section~\ref{sec:linforms}.
  It should be pointed out that, in the product on the right, every term occurs twice, as $x_{-Q}=x_Q$ for all $Q$.  
  
\begin{proposition}\label{prop:Pbound}
Let $E/\QQ$ be a quasi-minimal elliptic curve, $P\in E(\QQ)$ a point of inifinite order, and suppose that
 $nP$ is integral for some $n\geq 2$.  Then
$$\h(P)\leq \log n + \left(\frac{16}{3}M^2+2\right)h(E).$$
\end{proposition}

\begin{proof}
We recall Lemma 10.1 of \cite{david}, which states that if
 $Q\in E[n]\setminus\{\Ocal\}$, then
\begin{equation}\label{torsionbound}|x_Q|\leq 120n^2e^{h(E)},\end{equation}
where we are taking $E[n]$ here to represent the group of points of order dividing $n$ in $E(\CC)$.
Suppose that $nP$ is integral and that $$|x_P|>240n^2e^{32M^2 h(E)/3}\geq 240n^2e^{h(E)}\geq2|x_Q|$$
for all $Q\in E[n]\setminus\{\Ocal\}$.
  Then we have
$$|x_P-x_Q|>\frac{1}{2}|x_P|> 120n^2e^{32M^2 h(E)/3},$$
 for each $Q\in E[n]\setminus\{\Ocal\}$ and, as there are $n^2-1$ points in $E[n]\setminus\{\Ocal\}$, we obtain from \eqref{eq:divpoly} the bound
\begin{eqnarray*}
\log h_n^2&>& 2\log n+(n^2-1)\left(\frac{32}{3}M^2 h(E)+2\log n+\log 120\right)\\
&\geq&2\log n+8n^2M^2h(E)+(n^2-1)(2\log n+\log 120),
\end{eqnarray*}
as $n^2-1\geq\frac{3}{4}n^2$ for $n\geq 2$.
If $nP$ is integral, then $D_{n}=1$, and so we have by Lemma~\ref{twoeds} $$\log |h_n|\leq n^2M^2\log|\Delta(E)|< 4n^2M^2h(E),$$
as $\log|\Delta(E)|\leq 4h(E)$ (by the triangle inequality).
This, combined with the previous inequality, implies
$$n^2\log n+\frac{1}{2}\left(n^2-1\right)\log 120< 0,$$
however, which contradicts the assumption that $n\geq 2$.  Thus we have shown that \begin{equation}\label{eq:xbound}|x_P|\leq 240n^2e^{32M^2 h(E)/3}.\end{equation}
We note now that, in terms of the height above, Theorem~1.1 of \cite{silvcanon} implies that for all $P\in E(\QQ)$,
$$\left|\h(P)-\frac{1}{2}h(x_P)\right|<2h(E).$$
As $P$ is an integral point, $h(x_P)=\log |x_P|$, and so  we have
\begin{eqnarray*}\h(P)&\leq&\frac{1}{2}h(x_P)+2h(E)\\
&\leq& \log n+ \frac{16}{3}M^2h(E)+\frac{1}{2}\log 240\\
&\leq&\log n+ \left(\frac{16}{3}M^2+2\right)h(E).
\end{eqnarray*}
\end{proof}

At this point, we require a lower bound on $\h(P)$.  In general, it is conjectured by Lang (see \cite[p.~233]{silverman}) that
$$\h(P)\gg h(E)$$
whenever $P\in E(\QQ)$ is not a point of finite order, where the implied constant is absolute.  Lang's conjecture is not proven,
but the  lemma below follows directly from more general results of Silverman \cite{silvlower} and Hindry-Silverman \cite{hindrysilv}.

\begin{lemma}\label{langlem}
Let $E/\mathbb{Q}$ be a quasi-minimal elliptic curve, let $P\in E(\mathbb{Q})$ be a point of infinite order, and let $M=M(P)$ be as defined above.  Then
$$\h(P)\geq C_\lambda h(E)$$
for some $C_\lambda=C_\lambda(M)$.  Furthermore, we may take $C_\lambda(M)^{-1}=O(M^6)$.
\end{lemma}

\begin{proof} Let $\mathscr{D}$ denote the minimal discriminant of $E$.  Note that, while we may not have $|\Delta(E)|=\mathscr{D}$, as our curves may not be minimal at 2 or 3, it is certainly true that $\Delta(E)\mid 6^{12}\mathscr{D}$.  If $E'$ is a global minimal model of $E$, then the curve
$$E'':y^2=x^3-27c_4(E')x-54c_6(E')$$
(see \cite[p.~46]{silverman} for the notation) is isomorphic to both $E'$ and $E$, and one computes $\Delta(E'')=6^{12}\Delta(E')$.  As $E''$ is a short Weierstrass model of $E$ with integer coefficients, and $E$ is quasi-minimal, $\Delta(E)\mid\Delta(E'')$.

We write the canonical height as a sum of local heights, normalized as in \cite{hindrysilv}.  Write  $j(E)=\alpha/\beta$, where $\alpha$ and $\beta$ are coprime integers.  By Theorem~1.2 of \cite{hindrysilv},  we may choose a $1\leq b\leq (44M)^2$ such that
$$\lambda(bMP)\geq \frac{1}{24}(h(j)-\log|\beta|),$$
where $\lambda$ is the archimedean local height.
As $bMP\in E_0(\QQ_p)$ for every prime, the sum of the non-archimedean local heights is at least $\frac{1}{12}\log|\mathscr{D}|$.  In particular,
\begin{eqnarray*}
\h(P)&\geq& \frac{1}{24b^2M^2}(h(j)-\log|\beta|+2\log|\mathscr{D}|)\\
&\geq & \frac{1}{24b^2M^2}(h(j)+\log|\Delta(E)|-24\log 6).
\end{eqnarray*}
It is easy to show that $h(j)+\log|\Delta(E)|\geq h(E)-4\log 2$, and so
$$\h(P)\geq \frac{1}{24\cdot 44^2M^{6}}(h(E)-c),$$
where $c=28\log 2+24\log 3$.
If $h(E)\geq 2c$, we have the bound
$$\h(P)\geq \frac{1}{10^5M^{6}}h(E).$$
As there are only finitely many elliptic curves $E$ with $h(E)\leq 2c$, there exists an (effectively computable) constant $\delta>0$ such that $\h(P)\geq \delta h(E)$ for non-torsion points $P$ on these curves (with no dependence on $M$).  Thus we have
$$\h(P)\geq \min\left\{(10^5M^6)^{-1}, \delta\right\} h(E).$$
\end{proof}

For the main result of the section, we will need the following simple estimate.

\begin{lemma}\label{calculus}
Let $a, b>0$ be real numbers, and set $f(x)=x^2-a\log(x)-b$.  Then $f(x)\geq 0$ for $x\geq\max\{e, a+b\}$.
\end{lemma}


\begin{proposition}\label{prop:primes}
For all quasi-minimal $E/\QQ$ and non-torsion $P\in E(\QQ)$, there is a constant $c_0$ depending only on $M$, such that if $nP$ is integral and $n>c_0$, then $n$ is prime.  Furthermore, we may choose $c_0=O(M^{16})$, where the implied constant is absolute.
\end{proposition}

\begin{proof}
Suppose that $n$ is composite, that $nP$ is integral, and let $n=qa$, where $2\leq q\leq\sqrt{n}$ is prime.  Then $q(aP)$ is integral, and so by Proposition~\ref{prop:Pbound}, we have that
$$\h(aP)\leq\log q+\left(\frac{16}{3}M^2+2\right)h(E).$$
  We have, on the other hand, that
$$\hat{h}(aP)=a^2\hat{h}(P)\geq a^2C_\lambda h(E),$$
by Lemma~\ref{langlem}.
Thus, we have
$$a^2\leq \frac{\log q}{C_\lambda h(E)}+\frac{\left(\frac{16}{3}M^2+2\right)}{C_\lambda}.$$
As $q\leq a$, the above bounds $a$, and hence also $n\leq a^2$.  From Lemma~\ref{calculus} we have $$a\leq\max\left\{e, C_\lambda^{-1}\left(h(E)^{-1}+\left(\frac{16}{3}M^2+2\right)\right)\right\}=O(M^8).$$ \end{proof}


\section{Linear forms in elliptic logarithms}\label{sec:linforms}

In this section we will use David's explicit lower bounds on linear forms in elliptic logarithms to obtain a bound on $n$ such that $nP$ is integral.  Our bound will be of the form
$$n<Ch(E)^{5/2},$$
for some explicit constant $C=O(M^5\log^+(M)^{3/2})$.  Here, and throughout, we set
$$\log^+(x)=\max\{\log |x|, 1\}.$$  

Let $\omega$ be the real period of $E$, and consider the linear form
$$L_{n, m}(z, \omega)=nz+m\omega.$$
Let $z$ be the principal value of the elliptic logarithm of $P$, that is, the value in the fundamental parallelogram of the period lattice of $E$ such that $P=(\wp(z), \frac{1}{2}\wp'(z))$, and let  $m$ be chosen such that $L_{n,m}(z, \omega)$ is the principal value of the elliptic logarithm of $nP$. 
Then we will show, as in \cite{strtza}, that if $nP$ is integral, then $L_{n,m}(z,\omega)$ is very small.  The explicit results of David, on the other hand, give us a lower bound on the value of this form, given the upper bound on $\h(P)$ found in Proposition~\ref{prop:Pbound}.    

First, we must explicitly relate the elliptic logarithm to the naive archimedean height.  The following estimate is based on similar inequalities in \cite{strtza}, but it proved here for completeness.
\begin{lemma}\label{lem:elllogs}
Let $Q\in E(\QQ)$ such that
\begin{equation}\label{eq:xnpbig}x_{Q}\geq 2\max\{|x_T|:T\in E[2]\setminus\{\Ocal\}\}.\end{equation}
Then if $z$ is the principal value of the elliptic logarithm of $Q$,
$$-\frac{3}{2}\log 2\leq \log|z|+\frac{1}{2}\log|x_Q|\leq \frac{3}{2}\log 2$$
\end{lemma}

\begin{proof}
If our elliptic curve is written in short Weierstrass form, $$E:y^2=f(x)=\prod_{T\in E[2]\setminus\{\Ocal\}}(x-x_T),$$ then the elliptic logarithm satisfies
$$|z|=\left|\frac{1}{2}\int_{x_Q}^\infty  \frac{dt}{\sqrt{f(t)}}\right|.$$
In particular, if $$x_Q\geq 2\max\{|x_T|:T\in E[2]\setminus\{\mathcal{O}\}\},$$ then we have $|f(t)|\leq 8t^3$ for $t\geq x_Q$, and so
$$|z|\geq\left|\frac{1}{2}\int_{x_P}^\infty\frac{dt}{\sqrt{8t^3}}\right|\geq \frac{\sqrt{2}}{4}|x_P|^{-1/2}.$$
On the other hand, we have $|f(t)|\geq \frac{1}{8}t^3$ for $t\geq x_Q$, and so
$$|z|\leq\left|\frac{1}{2}\int_{x_P}^\infty\frac{2\sqrt{2}dt}{\sqrt{t^3}}\right|\leq 2\sqrt{2}|x_P|^{-1/2}.$$
Taking logarithms, we obtain the estimate in the lemma.
\end{proof}

The following lemma will also be used in Section~\ref{sec:gap} to examine the gaps between values of $n$ such that $nP$ is integral.

\begin{lemma}\label{lem:upper}
In the notation above,  there exist absolute positive constants $c_1$ and $c_2$ such that if   $nP$ is an integral point, and $n>c_2$, then
$$\log|L_{n,m}(z,\omega)|\leq -c_1n^2h(E).$$
Furthermore, we may take $c_1^{-1}=O(M^6)$ and $c_2=O(M^3)$.
\end{lemma}

\begin{proof}
From the previous lemma, we know that if
\begin{equation}\label{eq:xnpbigs}x_{nP}\geq 2\max\{|x_T|:T\in E[2]\setminus\{\Ocal\}\},\end{equation}
then
$$\log|L_{n,m}(z,\omega)|\leq \frac{3}{2}\log 2-\frac{1}{2}\log|x_{nP}|.$$
In this case, if $nP$ is integral, then $\log|x_{nP}|=h(x_{nP})$, and so, applying again Theorem~1.1 of \cite{silvcanon}
$$\log|L_{n,m}(z,\omega)|\leq -\frac{1}{2}h(x_{nP})+\frac{3}{2}\log(2)\leq -n^2\h(P)+3h(E).$$
Lemma~\ref{langlem} provides that
$\h(P)\geq\Lang h(E)$,
 and so
$$\log|L_{n,m}(z,\omega)|\leq (3-n^2\Lang)h(E).$$
This gives
$$\log|L_{n,m}(z,\omega)|\leq -\frac{\Lang}{2}n^2h(E)$$
for all
$n\geq \sqrt{6/\Lang}$.

Suppose, then, that  \eqref{eq:xnpbigs} fails.  Note that if $x_{nP} < -|x_T|$ for all points $T$ of exact order 2, then 
$$y_{nP}^2=\prod_{T\in E[2]\setminus\{\Ocal\}}(x_{nP}-x_T)<0,$$
a contradiction.  Thus if \eqref{eq:xnpbigs} fails, we have
$$|x_{nP}|\leq 2\max\{|x_T|:T\in E[2]\setminus\{\Ocal\}\}.$$
 We may appeal again to  Lemma 10.1 of \cite{david} (see \eqref{torsionbound} above) to obtain
\begin{eqnarray*}
h(x_{nP})=\log|x_{nP}|&\leq& \log\max\{|x_T|:T\in E[2]\setminus\{\Ocal\}\}+\log 2\\
&\leq & h(E)+6.174<6h(E).
\end{eqnarray*}
  Applying Lemma~\ref{langlem} and Theorem~1.1 of \cite{silvcanon} again, we have $$\Lang n^2h(E)\leq \h(nP)\leq \frac{1}{2}h(x_{nP})+2h(E)< 5h(E),$$
which implies that $n<\sqrt{5/\Lang}$.  This proves the lemma, with $c_1^{-1}=O(M^6)$ and $c_2=\sqrt{6/\Lang}=O(M^3)$.
\end{proof}

Lemma~\ref{lem:upper}, combined with David's explicit lower bounds on linear forms in elliptic logarithms, is sufficiently strong to give us a bound on $n$ such that $nP$ is integral that depends only on $M$ and the height of the elliptic curve in question, and in a predictable way.  The following lemma is a special case of Theorem~2.1 of \cite{david}, repeated here for the convenience of the reader.  Recall that our definition of $\h$ differs from that used in \cite{david} by a factor of 2.

\begin{lemma}\label{lem:david}
Let $E/\QQ$ be an elliptic curve, and let $\omega$ and $\omega'$ be the real and complex periods of $E$, chosen such that $\tau=\omega'/\omega$ is in the fundamental region 
$$\left\{z\in\CC:|z|\geq 1, \ \Im(z)>0, \text{ and }|\Re(z)|\leq\frac{1}{2}\right\}$$
of the action of $\operatorname{SL}_2(\ZZ)$ on the upper half plane.  Let $P$, $z$, and $L_{n, m}$ be defined as above, and let $B$, $V_1$, and $V_2$ be positive real numbers chosen such that
\begin{gather*}
\log(V_2)\geq\max\left\{h(E), \frac{3\pi}{\Im(\tau)}\right\}\\
\log(V_1)\geq\max\left\{2\h(P), h(E), \frac{3\pi |z|^2}{|\omega|^2\Im(\tau)}, \log(V_2)\right\}\\
\intertext{and}
\log(B)\geq\max\left\{eh(E), \log|n|, \log|m|, \log(V_1)\right\}.
\end{gather*}
Then either $L_{n, m}(z, \omega)=0$, or else
\begin{equation}\label{eq:lower}\log\left|L_{n,m}(z,\omega)\right|\geq -C(\log(B)+1)(\log\log(B)+h(E)+1)^3\log(V_1)\log(V_2),\end{equation}
where $C$ is some large absolute constant (we may take $C=4\times 10^{41}$).
\end{lemma}

We note that $L_{n, m}(z, w)$ cannot vanish if $P$ is a point of infinite order.

\begin{proposition}\label{prop:floatbound}
Let $E/\QQ$ be a quasi-minimal elliptic curve, and let $P\in E(\QQ)$ be a point of infinite order.
There exist positive constants $c_3$ and $c_4$ (depending only on $M$) such that for all $n>c_3$, $nP$ integral implies
$$n<c_4h(E)^{5/2}.$$
We may choose the constants such that $c_3, c_4=O(M^5\log^+(M)^{3/2})$.
\end{proposition}

\begin{proof}
We appeal to Lemma~\ref{lem:david}, and assume the notation used there.
Note that as $\Im(\tau)\geq\frac{\sqrt{3}}{2}$, we may take, for $h(E)\geq 2\pi\sqrt{3}$, $\log(V_2)=h(E)$.
In light of Proposition~\ref{prop:Pbound} and the observation that $|z|\leq\omega/2$, it suffices (under the assumption that $nP$ is integral) to choose any $V_1$ with $$\log(V_1)\geq 2\log n+(11M^2+4)h(E).$$  As  $|nz+m\omega|\leq\omega/2$, we have $|m|<n$, and so we may take $$\log(B)=\log(V_1)\geq 2\log n+(11M^2+4)h(E)$$ as well.

Suppose first that $\log(n)<h(E)$.  Then we may take $\log(B)=\log(V_1)=(11M^2+6)h(E)$, whence $$\log|L_{n, m}(z,\omega)|>-C'h(E)^6,$$
where $C'=O(M^4\log^+(M)^3)$.  Applying Lemma~\ref{lem:upper}, we have
$$n^2h(E) < \frac{C'}{c_1}h(E)^6,$$
and so $n< c_4h(E)^\frac{5}{2}$, for some constant $c_4=O(M^5\log^+(M)^{3/2})$.

If, on the other hand, $\log(n)\geq h(E)$, then Lemma~\ref{lem:upper} and \eqref{eq:lower} combine to produce a bound of the form
$$n^2<C''\log(n)^6,$$
where $C''=O(M^{10}\log^+(M)^3)$.  This again bounds $n$, by some term of the form $O(M^5\log^+(M)^{3/2})$.  Thus we obtain the result for elliptic curves $E/\QQ$ with $h(E)\geq 2\pi\sqrt{3}$.  For the remaining curves we may (effectively) find all integral points, and adjust the constants in the statement accordingly.
\end{proof}


\section{The multipliers grow rapidly}\label{sec:gap}

In Section~\ref{sec:eds} we saw that large values of $n$ such that $nP$ is integral must be prime.  In this section we suppose that there are two large values $n_1<n_2$ such that $n_iP$ is integral, and construct a function $f(x,y)$ such that $f(n_1, n_2)$ is very small.  An elementary lower bound on our function exists, and we use the primality of $n_1$ and $n_2$ to show that this lower bound does not vanish.
A bound of the form
$$n_1^2h(E)\leq c_5\log n_2$$
results.
In light of Proposition~\ref{prop:floatbound}, the above bounds $h(E)$.

Before proceeding with the proof of this, we need a lower bound on the elliptic logarithm of $P$ in order to show that, for sufficiently large $n$, the principal value of the elliptic logarithm of $nP$ cannot be $nz$, with $z$ the principal value of the elliptic logarithm of $P$.

\begin{lemma}\label{lem:mnotzero}
If $E/\QQ$ is a quasi-minimal elliptic curve, and $P\in E(\QQ)$ is a point of infinite order, then there is a constant $C(M)=O(M^4)$ such that the following holds: if $z$ is the principal value of the elliptic logarithm of $P$, $\omega$ is the real period of $E$, and  $nP$ is an integral point,  then either $|nz|>\omega/2$ or $n<C$.
\end{lemma}

\begin{proof}
In the course of proving Proposition~\ref{prop:Pbound}, we concluded that if $n\geq 2$ and $nP$ is an integral point, then
\begin{equation}\label{bound2}\log |x_P|\leq 2\log n+\frac{32}{3}M^2h(E)+\log 240\end{equation}
(this is \eqref{eq:xbound}).
If $|nz|\leq \omega/2$ (that is, if $nz$ is the principal value of the elliptic logarithm of $nP$), then we have, by Lemma~\ref{lem:upper}, 
\begin{equation}\label{bound1}|nz|\leq \exp(-c_1n^2h(E))\end{equation}
for $n>c_2$ (with the constants as in said lemma).
On the other hand, we may conclude from Lemma~\ref{lem:elllogs} that
 if $$x_P\geq 2\max\{|x_T|:T\in E[2]\setminus\{\mathcal{O}\}\},$$ then we have 
 $$-\log|z|\leq \frac{1}{2}\log |x_P|+\frac{3}{2}\log 2\leq \log n+\frac{16}{3}M^2h(E)+3.8,$$
by \eqref{bound2}.  But \eqref{bound1} ensures that
$$c_1n^2h(E)\leq -\log|z|-\log n,$$
and so $c_1n^2h(E)\leq \frac{16}{3}M^2h(E)+3.8$, bounding $n$.  As $c_1^{-1}=O(M^6)$, this bound is of the form $O(M^4)$.

If, on the other hand, $$x_P<\alpha=2\max\{|x_T|:T\in E[2]\setminus\{\mathcal{O}\}\},$$ (certainly $-\alpha<x_P$, as there are no points $Q\in E(\RR)$ with $x_Q\leq-\alpha$) then $z$ is bounded away from 0, the pole of the Weierstrass $\wp$-function, again contradicting \eqref{bound1}.  Specifically,
$$|z|\geq\left|\frac{1}{2}\int_{x_P}^\infty\frac{dt}{\sqrt{f(t)}}\right|\geq\left|\frac{1}{2}\int_{\alpha}^\infty\frac{dt}{\sqrt{f(t)}}\right|\geq \frac{\sqrt{2}}{4}\alpha^{-1/2},$$
and so
$$-\log|z|\leq\frac{1}{2}\log \alpha+\frac{3}{2}\log 2\leq \frac{1}{2}h(E)+4.5$$
by \eqref{torsionbound} applied with $n=2$.
Comparing again with \eqref{bound1} bounds $n$ by something of the form $O(M^3)$.
\end{proof}

\begin{proposition}\label{prop:gap}
Let $E/\QQ$ be quasi-minimal, and let $P\in E(\QQ)$ be a point of infinite order.  Suppose that $n_2P$ and $n_1P$ are integral points.  Then there exist constants $c_5=O(M^6)$ and $c_6=O(M^{16})$,
$$n_1^2h(E)\leq c_5\log n_2$$
if $n_1, n_2>c_6$.
\end{proposition}

\begin{proof}
Let $z\in\CC$, $\omega>0$, and $m_1$ and $m_2$ be chosen above, i.e., so that 
$$L_{n_i, m_i}(z, \omega)=|n_iz+m_i\omega|\leq\frac{\omega}{2}.$$
We have, by Lemma~\ref{lem:upper}, 
$$\left|n_iz+m_i\omega\right|\leq \exp\left(-c_1n_i^2h(E)\right)$$
for $i=1, 2$, if $n_i>c_2$.  By the triangle inequality,
\begin{eqnarray*}
\omega\left|n_2m_1-n_1m_2\right|&\leq& n_2|n_1z+m_1\omega|+n_1|n_2z+m_2\omega|\\&\leq& n_2\exp\left(-c_1n_1^2h(E)\right)+n_1\exp\left(-c_1n_2^2h(E)\right).\end{eqnarray*}
Now suppose that $n_1$ and $n_2$ are greater than $c_0$, the constant from Proposition~\ref{prop:primes}, so that both must be prime, and suppose that $n_2m_1=n_1m_2$.  As $n_1$ and $n_2$ are distinct primes, we have $n_1\mid m_1$.  If we assume that $n_1>C$, where $C$ is the constant in Lemma~\ref{lem:mnotzero}, then $m_1\neq 0$, and so $n_1\leq |m_1|$.  But $|z|\leq \omega/2$ and $|n_1z+m_1\omega|\leq\omega/2$, and so 
$$2|m_1|\leq\frac{2}{\omega}\left(|n_1z+m_1\omega|+|n_1z|\right)\leq n_1+1,$$
giving the rather unlikely inequality $2n_1\leq n_1+1$.

Thus we have $n_2m_1-n_1m_2\neq 0$, and so
$$\omega\leq n_2\exp\left(-c_1n_1^2h(E)\right)+n_1\exp\left(-c_1n_2^2h(E)\right).$$
As one of the terms on the right must exceed the average of the two,
$$\frac{\omega}{2}\leq n_2\exp\left(-c_1n_1^2h(E)\right)$$
(or the same with indices reversed, but recall that $n_1<n_2$).
This yields
$$c_1n_1^2h(E)+\log(\omega)-\log(2)\leq \log n_2,$$
and so it suffices to show that
$$-\log \omega\ll h(E)$$
(where the implied constant is independent of $M$).
Note that if $E:y^2=x^3+Ax+B$, then
$$\omega=\int_{x_Q}^{\infty}\frac{dt}{\sqrt{t^3+At+B}}$$
where $Q\in E[2]$ is chosen so that $x_Q$ is the largest of the real roots of  $t^3+At+B$.
If $x_Q\geq 1$, then
$$\omega=\int_{x_Q}^\infty\frac{dt}{\sqrt{t^3+At+B}}\geq (x_Q(1+|A|+|B|))^{-1/2}.$$
If $x_Q<1$, then
$$\omega\geq\int_1^\infty\frac{dt}{\sqrt{t^3+At+B}}\geq (1+|A|+|B|)^{-1/2}.$$
Either way we have our bound.

Taking $n_1$ large enough, we have
$$n_1^2h(E)\leq \frac{2}{c_1}\log n_2.$$
We may take $c_6$ to be the larger of the constant $c_0=O(M^{16})$ from Proposition~\ref{prop:primes}, $C=O(M^4)$ from Lemma~\ref{lem:mnotzero}, and the (absolute) constant required to ensure that $\log n_1 h(E)\geq 2\log(2/\omega)$.
\end{proof}
We may now proceed with the proof of the main result.

\begin{proof}[Proof of Theorem~\ref{th:main}.]
Let $E/\QQ$ be a quasi-minimal elliptic curve, and $P\in E(\QQ)$ be an integral point of infinite order.
In the notation above, let $$C_0=\max\{c_0, c_3, c_6, M^{16}\}=O(M^{16})$$ be the larger of the constants appearing in Propositions~\ref{prop:primes}, \ref{prop:floatbound}, and \ref{prop:gap}.  Then if  $n_1P$ and $n_2P$ are both integral, where $n_i>C_0$, we have
$$n_1^2 h(E)\leq c_5\log n_2\quad\text{ and }\quad n_2\leq c_4 h(E)^{5/2}.$$
Combining these, we obtain
\begin{eqnarray*}
h(E)&\leq &\frac{5c_5}{2n_1^2}\left(\log h(E)+\log c_4\right)\\
&\leq&O(\log(h(E))),
\end{eqnarray*}
where the implied constant is absolute,
as $n_1>C_0$ (indeed, as $n_1\geq M^{16}$).  The above provides a (uniform) bound $h(E)\leq N$.  Thus for $E/\QQ$ with $h(E)> N$, there can be at most one $n> C_0$ such that $nP$ is integral.
Let
$$C_0'=\sup_{h(E)\leq N}\{n:nP\text{ is integral for some }P\in E(\QQ)\},$$
the maximum of a finite set (that can be effectively computed if $N$ is explicitly known).  Then with $C=\max\{C_0, C_0'\}=O(M^{16})$, the above claim holds.
\end{proof}

\section{Quadratic Twists and Congruent Number Curves}\label{sec:examples}

As mentioned in the introduction, one can state much more explicit results for families of  twists, in part because the quantity $M(P)$ is bounded in this setting (as $M$ may be bounded in terms of $j(E)$, which is invariant in a family of twists).  It is the existence of a lower bound on heights on points on a family of twists that makes Theorem~\ref{th:main} stronger in this setting, but it is Theorem~8 of \cite{meandjoe}, a generalization of Theorem~3 of \cite{me:eds}, that allows us to, modulo some computation, give a very small value for the constant in the theorem.  These results concern primitive divisors in elliptic divisibility sequences; we state below a theorem that follows immediately from Theorem~8 of \cite{meandjoe}.  Fix an elliptic curve $$E:y^2=x^3+Ax+B$$ as above, and consider quadratic twists $$E':y^2=x^3+Ad^2x+Bd^3$$ of  $E$ which are quasi-minimal (that is, $d\in\ZZ$ square free).  
\begin{theorem}[Ingram-Silverman \cite{meandjoe}]\label{th:mandj}
Fix an integer $n\geq 3$.  Then there exist at most finitely many twists $E'$ of $E$ and non-torsion points $P\in E'(\QQ)$ such that $nP$ is integral.  Furthermore, one may effectively find all such points on all such twists.
\end{theorem}

The effective computation alluded to turns out to be the resolution of a Thue-Mahler equation that depends on $n$ and $E$, and   one can in fact replace `integral' here with `$S$-integral', for any fixed, finite set of primes $S$.  The theorem can also be made quantitative  by a result of Bombieri \cite{bombieri}.  The following result is immediate from Theorem~\ref{th:main} and Theorem~\ref{th:mandj}.

\begin{proposition}\label{prop:twists}
Fix an elliptic curve $E/\QQ$.  Then for quasi-minimal twists $E'$ of $E$ of sufficient height we have the following: for each $P\in E'(\QQ)$ there is at most one integer $n\geq 3$ such that $nP$ is integral.
\end{proposition}

\begin{remark}
We shall see below that the condition $n\geq 3$ can be relaxed to $n\geq 2$ for congruent number curves.  To see that this is not the case in general, consider that there are infinitely many integral points $P$ on minimal Mordell curves such that $2P$ is also integral.
One may demonstrate this by applying a result of Erd\H{o}s \cite{erdos} to show that the polynomial $1-8u^3$ takes infinitely many square-free values, as $u$ ranges over $\ZZ$.  Let $M$ be one of these values.  Then $E:y^2=x^3+M$ is a minimal Mordell curve, and the  double of the integral point $(2u, 1)$ on $E$ is $(4u(9u^3-1), -216u^6+36u^3-1)$, itself an integral point.
\end{remark}

\begin{corollary}
Fix an elliptic curve $E/\QQ$.  Then for all quasi-minimal twists $E'$ of $E$ of sufficient height and all torsion-free subgroups $\Gamma\subseteq E'(\QQ)$ of rank one, $\Gamma$ contains at most 6 (affine) integral points.
\end{corollary}

\begin{proof}
If $\Gamma\subseteq E'(\QQ)$ is torsion-free and has rank one, then $\Gamma$ consists only of the points $nP$, $n\in\ZZ$, for some $P\in E'(\QQ)$.  By the proposition, there is at most one $n\geq 3$ such that $nP$ is integral (taking $h(E')$ large enough), and so the possible integral points in $\Gamma$ are at most $\{\pm P, \pm 2P, \pm kP\}$ for this one value of $k$.
\end{proof}

In particular, if $E$ has no $\QQ$-rational  points of order 2, then all but finitely many twists $E'$ will be torsion-free, and so we may take $\Gamma=E'(\QQ)$ if $\operatorname{rank}(E'/\QQ)=1$.

We now proceed with the proof of Theorem~\ref{th:cong}.
Note that we may conclude from Lemma~6 of \cite{me:eds} that $3P$, $5P$, and $7P$ cannot be integral for $P\in E_N(\QQ)$.  In light of this, and the following lemma, we may assume that if $nP$ is integral, then $n$ isn't divisible by 2, 3, 5, or 7, and hence $n\geq 11$.

\begin{lemma}\label{lem:timestwo}
Let $N$ be square free, and let $P\in E_N(\QQ)$ be a point of infinite order.  Then $2P$ is not integral.
\end{lemma}

\begin{proof}
The conclusion is immediate if $P$ is not itself integral, so suppose $x_P\in\ZZ$.
We will show that $\ord_2(x_{2P})<0$.
Note that
$$x_{2P}=\frac{(x_P^2+N^2)^2}{4(x_P^3-N^2x_P)},$$
which is clearly not integral unless, perhaps, $x_P\equiv N\MOD{2}$.  Suppose, first, that $x_P$ and $N$ are both odd.  Then we have $x_P^2\equiv N^2\equiv 1\MOD{4}$, and so $\ord_2((x_P^2+N^2)^2)=2$.  On the other hand, $x_P + N$ and $x_P-N$ are both even, and so $\ord_2(4(x_P^3-N^2x_P))\geq 4$, producing $\ord_2(x_{2P})\leq -2$.

 Now suppose that $x_P$ and $N$ are both even, and write $x_P=2x_1$, $N=2N_1$ noting that, as $N$ is squarefree, $N_1$ must be odd.  We have
$$x_{2P}=\frac{(x_1^2+N_1^2)^2}{2(x_1^3-N_1^2x_1)}.$$
If $x_1$ is even, then $x_1^2+N_1^2$ is odd and we are finished, so suppose that $x_1$ is odd.  Again we have $\ord_2((x_1^2+N_1^2)^2)=2$, while $\ord_2(2(x_1^3-N_1^2x_1))\geq 3$.  This shows that $\ord_2(x_{2P})\leq -1$, and proves the lemma.
\end{proof}

We are now in a position to prove Theorem~\ref{th:cong}, following the line of reasoning presented in the proof of Theorem~\ref{th:main}.  The proof consists of a series of claims, which are strong forms of various lemmas above.  

\begin{proof}[Proof of Theorem~\ref{th:cong}]
Throughout we will take $N$ to be square free (it is a simple matter to construct counter-examples to Theorem~\ref{th:cong} if we allow $N$ to have square divisors), and $P$ will range over (non-torsion) integral points on $E_N(\QQ)$.  We will denote by $\omega_N$ the real period of $E_N$, noting that $\omega_N=N^{-1/2}\omega_1$.
In order to minimize the amount of computation required, we will optimize the entire argument for the current setting.  We will use the strong estimates from \cite{bst} (which make explicit the more general techniques of \cite{silvcanon}), which state that for any $P\in E_N(\QQ)$ of infinite order,
\begin{gather}
-\frac{1}{2}\log N-\frac{1}{4}\log 2\leq \h(P)-\frac{1}{2}h(x_P)\leq \frac{1}{4}\log(N^2+1)+\frac{1}{12}\log 2,\label{eq:conghdiff}\\
\h(P)\geq\frac{1}{16}\log(2N^2),\label{eq:conglang}\\
\intertext{and, if $P$ is an integral point on the identity component $E^0_N(\RR)$ of $E_N(\RR)$ (i.e., the connected component of the real locus of $E$ which contains the point at infinity), then}
\h(P)\leq\frac{1}{2}h(x_P)+\frac{1}{3}\log 2\label{eq:bstupper}.
\end{gather}
In order to exploit \eqref{eq:bstupper}, we must first treat integral points which reside on the non-identity component of $E$.  This turns out to be fairly simple, as this component is bounded at the archimedean place, giving a strong form of Siegel's Theorem trivially.  Note that the following claim follows directly from results in \cite{me:eds}, but is proven here for completeness.
\begin{claim}\label{cl:trivialcomp}
Suppose that $nP$ is an integral point on the non-identity component of $E_N(\QQ)$, with $n\geq 1$.  Then $n=1$.
\end{claim}
\begin{proof}
If $nP\in E_N(\QQ)\setminus E_N^0(\QQ)$ is an integral point (of infinite order), then we have immediately that $-N< x_{nP}< 0$.  By \eqref{eq:conglang}, we obtain
\begin{eqnarray*}
\frac{n^2}{16}\log(2N^2)&\leq&\h(nP)\\
&\leq& \frac{1}{2}\log|x_{nP}|+\frac{1}{4}\log (N^2+1)+\frac{1}{12}\log 2\\
&<&\frac{1}{2}\log N+\frac{1}{4}\log(N^2+1)+\frac{1}{12}\log 2,
\end{eqnarray*}
whence
$$n^2 < 8\left(\frac{\frac{1}{2}\log N+\frac{1}{4}\log(N^2+1)+\frac{1}{12}\log 2}{\log N+\frac{1}{2}\log 2}\right)\leq 8,$$
for $N\geq 1$.  As $2P$ cannot be integral, we are done.
\end{proof}
Note that $E_N(\QQ)/E_N^0(\QQ)\cong\ZZ/2\ZZ$, and so if $nP\in E_N^0(\QQ)$, we must have either $P\in E_N^0(\QQ)$ or $2\mid n$.  If $nP$ is an integral point, Lemma~\ref{lem:timestwo} precludes the second option, and so we will from this point forward assume that $P$, and hence any multiple of $P$, is on the trivial connected component of $E$.

Our next claim is a sharper version of Lemma~10.1 of \cite{david} (compare with \eqref{torsionbound} above).
\begin{claim}\label{cl:congtwotorsbound}
Let $Q\in E_N(\CC)$ be a point of order dividing $n$ (other than the identity).  Then $|x_Q|\leq \frac{1}{2}n^2N$.
\end{claim}
\begin{proof}
In light of the isomorphism
\begin{alignat*}{1}
E_N(\mathbb{C})&\to E_1(\mathbb{C})\\
(x, y)&\mapsto (xN^{-1}, yN^{-3/2}),
\end{alignat*}
it suffices to prove the claim in the case $N=1$.  We appeal to
another isomorphism to prove our result. Let $\Lambda=\omega_1\ZZ[i]$ be the period lattice of $E_1$.  For the purpose of the estimates below, we will note that $2.62<\omega_1<2.63$.  Then if $\wp$ is the Weierstrass function
$$\wp(z)=\frac{1}{z^2}+\sum_{\substack{u\in\Lambda\\ u\neq 0}}\left(\frac{1}{(u-z)^2}-\frac{1}{u^2}\right),$$
  it is well known that
\begin{alignat*}{1}
\mathbb{C}/\Lambda&\to E_1(\mathbb{C})\\
z&\mapsto \left(\wp(z), \frac{1}{2}\wp'(z)\right),
\end{alignat*}
is an isomorphism.  Our result is essentially the observation that, near $z=0$, $|\wp(z)|=|z|^{-2}+O(1)$, but we wish to make this explicit.

Note that we may choose a representative $z$ of any class in
$\mathbb{C}/\Lambda$ such that $|\Re(z)|, |\Im(z)|\leq \omega_1/2$.  If
we do so, we have $|u-z|\geq |u|/2$ for all $u\in\Lambda$, and so
$$\left|\sum_{\substack{u\in\Lambda\\ u\neq 0}}
\left(\frac{1}{(u-z)^2}-\frac{1}{u^2}\right)
\right|\leq 2|z|\sum_{\substack{u\in\Lambda\\
  u\neq 0}} \frac{4}{|u|^3}+|z|^2\sum_{\substack{u\in\Lambda\\
  u\neq 0}}\frac{4}{|u|^4}$$
For $s>1$, let
$$F(s)=\sum_{\substack{u\in\Lambda\\
  u\neq0}}|u|^{-2s}=\frac{1}{\omega_1^{2s}}\sum_{\substack{i, j\in\mathbb{Z}\\ i^2+j^2\neq 0}}(i^2+j^2)^{-s}.$$
We have
$$\frac{\omega_1^{2s}}{4}F(s)=\sum_{i, j=1}^\infty(i^2+j^2)^{-s}+\zeta(2s),$$
where $\zeta$ is the Riemann zeta function.  Note that
$$\sum_{\substack{i, j\geq0 \\ i^2+j^2\geq 1}}((i+1)^2+(j+1)^2)^{-s}\leq\frac{1}{4} \underset{\{x^2+y^2\geq 1\}}{\iint}(x^2+y^2)^{-s}\leq \frac{\pi}{4(s-1)},$$
and so $$\frac{\omega_1^{2s}}{4}F(s)  \leq 2^{-s}+\frac{\pi}{4(s-1)}+\zeta(2s),$$
 and thus
$$|F(3/2)|\leq 0.694 \text{ and }\
|F(2)|\leq 0.180.$$
Now, as $|z|\leq \omega_1/\sqrt{2}$, we have
$$|\wp(z)|\leq |z|^{-2}+8F(3/2)|z|+4F(2)|z|^2\leq |z|^{-2}+12.755.$$
If $z\in\mathbb{C}/\Lambda$ is a  point of order dividing $n$
(other than $0$),
$|z|\geq \frac{\omega_1}{n}$, and so
$$|\wp(z)|\leq \frac{n^2}{\omega_1^2}+12.755\leq \frac{n^2}{2}$$
so long as $n\geq 6$.  The  cases $2\leq n\leq 5$ may be checked by
explicitly computing the relevant torsion points in $E_1(\CC)$.
\end{proof}
\begin{claim}\label{cl:congdivisbound}
Suppose  that $nP$ is an integral point, and let $h_n$ be defined as above.  Then
$$|h_n|\leq (2N)^{(n^2-1)/2}.$$
\end{claim}
\begin{proof}
We may, by Lemma~\ref{lem:timestwo}, assume that $n$ is odd, or else the statement is vacuously true.  Throughout we will make reference to the polynomials $\psi_n$ and $\phi_n$ defined in \cite[p.~105]{silverman}, and in Section~\ref{sec:eds}.
Recall that that
$$x_{nP}=\frac{\phi_n(P)}{\psi_n^2(P)}.$$
  To prove the result, we must first note that, for $n$ odd, $h_n=\psi_n(x, N)$ is a binary form in $x$ and $N$, as is $\phi_n$.  Again, we may assume that $x_P=a\in\ZZ$.  Suppose that $l\neq 2$ is a prime dividing $\psi_n(a, N)$.  Then $l\mid\phi_n(a, N)$ (or $nP$ isn't integral), and by the claim preceding Lemma~5 of \cite{me:eds} we have $l\mid N$; say $N=lN_1$. As $\phi_n(x, N)$ is monic in $x$, we must also have $l\mid a$, and we may write $a=la_1$.    Note that
$$a^3-N^2a=l^3(a_1^3-N_1^2a_1)$$
is a square, and so $l\mid a_1$ or $a_1\equiv \pm N_1\MOD{l}$.  As $N$ is square free, $N_1$ is not divisible by $l$, and so $l\mid a_1$ implies $l$ is not a divisor of $\psi_n(a_1, N_1)$.  If, on the other hand, $a_1\equiv\pm N_1\MOD{l}$, the aforementioned claim in \cite{me:eds} implies that
$$\psi_n(a_1, N_1)\equiv a_1^{(n^2-1)/2}\psi_n(1, \pm 1)\equiv \pm (2a_1)^{(n^2-1)/2}\MOD{l},$$
and so again $l$ does not divide $\psi_n(a_1, N_1)$.  Thus, for any prime $l\nmid 2n$,
$$\mathrm{ord}_l(\psi_n(a, N))=\left(\frac{n^2-1}{2}\right)\mathrm{ord}_l(\gcd(a, N)).$$


 The order of $2$ dividing $\psi_n$ may be obtained through a simple induction.  First note that, as $\psi_n(1,0)=n$, and $\psi_n(0, 1)=\pm 1$ when $n$ is odd, we cannot have $2\mid h_n$, unless $a$ and $N$ have the same parity.  Suppose that $a$ and $N$ are both odd.  We will prove by induction that
\begin{alignat*}{2}\ord_2(h_n)=&\frac{n^2-1}{4}\qquad&\text{if }n\text{ is odd}\\
\ord_2(h_n)\geq&\frac{n^2}{4}+\ord_2(b)\qquad&\text{if }n\text{ is even.}\end{alignat*}
We may check this from the definition for $h_1$ and $h_2$.  We may check this as well for $h_3$ and $h_4$ by simply computing all possible values of  $h_3$ and $h_4$ (in terms of $a$ and $N$) modulo $2^3$ and $2^{6+\ord_2(b)}$ respectively.  Now, if $n=2m+1$ is odd, we have (see, for example, \cite[p.~105]{silverman})
$$h_n=h_{m+2}h_m^3-h_{m-1}h_{m+1}^3.$$
If we suppose that the formula above holds for $h_j$ with $j<n$, then we have two cases to consider.  First, if $m$ is odd, then
\begin{gather*}
\ord_2(h_{m+2}h_m^3)=\frac{(m+2)^2-1}{4}+3\frac{(m^2-1)}{4}=\frac{(2m+1)^2-1}{4}\\
\ord_2(h_{m-1}h_{m+1}^3)\geq \frac{(2m+1)^2-1}{4}+1+4\ord_2(b).
\end{gather*}
As the latter is clearly greater than the former, we must have $\ord_2(h_n)=\frac{n^2-1}{4}$.  If, on the other hand, $m$ is even, we obtain
\begin{gather*}
\ord_2(h_{m+2}h_m^3)\geq \frac{(2m+1)^2-1}{4}+1+4\ord_2(b)\\
\ord_2(h_{m-1}h_{m+1}^3)=\frac{(m-1)^2-1}{4}+3\frac{(m+1)^2-1}{4}=\frac{(2m+1)^2-1}{4}.
\end{gather*}
Again we have $\ord_2(h_n)=\frac{n^2-1}{4}$.  Now we must establish the formula for $n=2m$ even, supposing that it holds for $h_k$ with $k<n$.  In this case, we have
$$h_2h_n=h_m(h_{m+2}h_{m-1}^2-h_{m-2}h_{m+1}^2).$$
If we suppose, first, that $m$ is even, we have
\begin{gather*}
\ord_2(h_{m+2}h_{m-1}^2)\geq\frac{3m^2}{4}+1+\ord_2(b)\\
\ord_2(h_{m-2}h_{m+1}^2)\geq\frac{3m^2}{4}+1+\ord_2(b).
\end{gather*}
As $h_2=2b$, we have
\begin{eqnarray*}
\ord_2(h_n)&=&\ord_2(h_m)-\ord_2(h_2)+\ord_2(h_{m+2}h_{m-1}^2-h_{m-2}h_{m+1}^2)\\
&\geq& \frac{4m^2}{4}+\ord_2(b).
\end{eqnarray*}
Finally, if $m$ is odd, 
\begin{gather*}
\ord_2(h_{m+2}h_{m-1}^2)\geq \frac{3m^2}{4}+\frac{5}{4}+2\ord_2(b)\\
\ord_2(h_{m-2}h_{m+1}^2)\geq\frac{3m^2}{4}+\frac{5}{4}+2\ord_2(b).
\end{gather*}
It follows, again because $\ord_2(h_2)=\ord_2(b)+1$, that
\begin{eqnarray*}
\ord_2(h_n)&=&\ord_2(h_m)-\ord_2(h_2)+\ord_2(h_{m+2}h_{m-1}^2-h_{m-2}h_{m+1}^2)\\
&\geq& \frac{4m^2}{4}+\ord_2(b).
\end{eqnarray*}
As $\ord_2(2m)=1$, we are done.

If, on the other hand, $a$ and $N$ are both even, we may reduce to essentially the previous case by writing $a=2a_1$ and $N=2N_1$.  Note that $N_1$ must be odd, and $h_n=2^{(n^2-1)/2}\psi_n(a_1, N_1)$ when $n$ is odd.  If $a_1$ is even, then, we have $\ord_2(h_n)=\frac{n^2-1}{2}$.  If $a_1$ is odd, then an induction similar to that above shows that
\begin{alignat*}{3}
\ord_2(h_n)=&\frac{3(n^2-1)}{4}&\text{if }n\text{ is odd}\\
\ord_2(h_n)\geq&\frac{3n^2}{4}+\ord_2(b)\quad&\text{if }n\text{ is even.}
\end{alignat*}
It follows that, for $n\geq 3$ odd, $$\ord_2(h_n)\leq\frac{n^2-1}{4}+\frac{n^2-1}{2}\ord_2(N)< \frac{n^2-1}{2}\ord_2(2N).$$
\end{proof}

\begin{claim}\label{cl:logqbound}
Suppose that $nP$ is an integral point, and $n\geq 2$.  Then
\begin{equation}\label{eq:conglogqbound}\h(P)\leq \log n+\frac{1}{2}\log N +\frac{1}{3}\log 2.\end{equation}
\end{claim}

\begin{proof} Again we need only concern ourselves with the case where $n$ is odd.
Suppose that $x_P\geq n^2N$, so that if $Q$ is a point of order
$n$ in $E_N(\mathbb{C})$,
$$|x_P-x_Q|>\frac{1}{2}x_P$$
by Claim~\ref{cl:congtwotorsbound}.
Then, if $\psi_n$ is the $n$-division polynomial for $E_N$, we have, by the formula~\eqref{eq:divpoly}
$$|h_n|=|\psi_n(x_P)|\geq \left(\frac{x_P}{2}\right)^\frac{n^2-1}{2}.$$
On the other hand, if $nP$ is integral then by Claim~\ref{cl:congdivisbound} we have
$$|\psi_n(x_P)|\leq (2N)^\frac{n^2-1}{2}.$$
Thus it is shown that $n^2N\leq |x_P|\leq 4N$, a contradiction.
So we have $-N<x_P< n^2N$.  The result follows from \eqref{eq:bstupper}.
\end{proof}

We will now prove an explicit statement of Proposition~\ref{prop:primes}, which in this setting has a much nicer form.
\begin{claim}\label{cl:primes}
Suppose that  $nP$ is integral, $n\geq 2$.  Then $n$ is prime.
\end{claim}
\begin{proof}
Suppose, to the contrary, that $n$ is composite, let $q$ be the smallest prime divisor of $n$, and set $a=n/q$.  Then $nP=q(aP)$ is integral, and so by Claim~\ref{cl:logqbound} we have
$$\h(aP)\leq\log q+\frac{1}{2}\log N+\frac{1}{3}\log 2.$$
On the other hand,
$$\h(aP)\geq\frac{a^2}{16}\log(2N^2)$$
by \eqref{eq:conglang}, and so
$$a^2\leq16\left(\frac{\log q+\frac{1}{2}\log N+\frac{1}{3}\log 2}{\log(2N^2)}\right).$$
Note that for $N< 5$, $E_N(\QQ)$ has no points of infinite order.  Thus we may assume $N\geq 5$, and the above yields
\begin{equation}\label{abound}q^2\leq a^2\leq 4.1\log q+4.217.\end{equation}
Lemma~\ref{calculus} allows us to conclude that $q\leq 8.317$, and checking the smaller values shows that $q\leq 2$. But $n$ must be odd, and so we have a contradiction.
\end{proof}

We will now derive an upper bound on $n$ such that $nP$ is integral, in terms of $\log N$.  We will assume that $N\geq 56$ to ensure that $h(E_N)=\log(4N^2)$, and later that $h(E_N)\geq 3\pi$.  The cases $N\leq 55$ will be treated below. 

\begin{claim}\label{cl:zsmall}
Suppose $nP$ is an integral point, and let $L_{n,m}=nz+m\omega$ be the principal value of the elliptic logarithm of $nP$, as above.  Then if $n\geq 2$,
\begin{equation}\label{eq:congupper}\log|L_{n,m}|\leq-\frac{n^2}{8}\log N.\end{equation}
\end{claim}

\begin{proof}
By Lemma~\ref{lem:elllogs}, we have
$$\log|L_{n,m}|\leq \frac{3}{2}\log 2-\frac{1}{2}\log|x_{nP}|$$
unless $x_{nP}\leq 2N$.  We have already treated the case of integral points on the non-identity component of $E_N(\QQ)$, and so if $x_{nP}<2N$ we have $N<x_{nP}<2N$.  From this it follows that
$$\frac{n^2}{16}\log(2N^2)\leq\h(nP)\leq \frac{1}{2}\log(2N)+\frac{1}{3}\log 2$$
by \eqref{eq:conglang} and \eqref{eq:bstupper}, from which we immediately conclude that $n\leq 2$.  As $2P$ cannot be integral, we ascertain that $x_{nP}>2N$, and so the above bound on $L_{n,m}$ holds.
The result now follows by observing that $h(x_{nP})=\log|x_{nP}|$, and so by \eqref{eq:bstupper} and \eqref{eq:conglang} respectively,
\begin{eqnarray*}-\frac{1}{2}\log|x_{nP}|+\frac{3}{2}\log 2&\leq&-\h(nP)+\frac{11}{6}\log 2\\
&\leq& -\frac{n^2}{16}\log(2N^2)+\frac{11}{6}\log 2\\
&\leq&-\frac{n^2}{8}\log N
\end{eqnarray*}
for $n\geq 6$.  As $nP$ cannot be integral for $n$ divisible by 2, 3, or 5, the claim holds.
\end{proof}

Before proceeding with our next claim, we require a simple estimate from calculus, which refines Lemma~\ref{calculus}.

\begin{claim}\label{cl:calculus}
Let $P\in\RR[x]$ be a polynomial of degree $d$, and denote by $P^{(k)}$ the $k$th derivative of $P$. Suppose that for some $W>0$ and every $0\leq k\leq d$ we have
$$W^2>2^{-k}P^{(k)}(\log W).$$
Then $x^2>P(\log x)$ for all $x\geq W$.
\end{claim}

\begin{proof}
Let $f(x)=x^2-P(\log x)$, so that our aim is to show that $f(x)>0$ for all $x\geq W$, where $W$ is as in the statement of the result.  Since we know that $f(W)>0$, it is sufficient to show that $f'(x)>0$ for all $x\geq W$, as $f(x)\leq f(W)$ for some $x>W$ would imply $f'(y)=0$ for some $y>W$, by Rolle's Theorem.  Note that  the condition $f'(x)>0$, for $x>0$, is equivalent to
$$x^2-\frac{1}{2}P'(\log x)>0.$$
Proceeding by induction, we see that it suffices to show, for any $m$, that
$$W^2>2^{-k}P^{(k)}(\log W)$$
for all $0\leq k\leq m$, and
$$x^2>2^{-(m+1)}P^{(m+1)}(\log x)$$
for all $x\geq W$.  But the last condition is automatic if we select $m=d$, as $P^{(d+1)}=0$.
\end{proof}

\begin{claim}\label{cl:linformsbound}
Suppose that $nP$ is integral.  Then
$$n\leq\max\{3.6\times 10^{27}, 9.196\times 10^{23}(\log N)^{5/2}\}.$$
\end{claim}

\begin{proof}
We proceed by estimating linear forms in elliptic logarithms as in Proposition~\ref{prop:floatbound}, appealing again the Lemma~\ref{lem:david}, and apoting the notation used there.
 Note that $\tau=i$ for all congruent number curves, and we will assume that $N\geq 56$ so that $h(E)=\log(4N^2)>3\pi$.
Then we have, by Claim~\ref{cl:zsmall},
\begin{eqnarray}\frac{n^2}{8}\log N&\leq& -\log|L_{n,m}|\nonumber\\&\leq& C(\log(B)+1)(\log\log (B)+\log(4N^2)+1)^3\log(V_1)\log(V_2),\label{eq:conglin}\end{eqnarray}
where $C=4\times 10^{41}$, if $B$, $V_1$, and $V_2$ are chosen as in Lemma~\ref{lem:david}.

Using  Claim~\ref{cl:logqbound} to bound $\h(P)$ from above (under the hypothesis that $nP$ is integral), we may set
\begin{gather*}
\log(V_1)=3\log\max\{ n,N\}+\frac{2}{3}\log 2\\
\log(B)=2e\log\max\{ n,N\}+2e\log 2.
\end{gather*}
To simplify matters, we will consider two cases.  First suppose that $N<n$.  In this case, we will use the assumption that $\log N>\log 56$ and the trivial estimate $$\log(\log
 n+\log 2)<\log n,$$ for $n\geq 2$, to obtain from  \eqref{eq:conglin} the bound
\begin{equation}\label{eq:nplogn}n^2\leq P(\log n),\end{equation}
where
$$P(x)=\frac{2592eC}{\log 56}\left(x+\log 2+\frac{1}{2e}\right)\left(x+ \log 2+\frac{1}{3}\right)^3\left(x+\frac{2}{9}\log 2\right)(x+\log 2).$$
One may check that, if $W=3.6\times 10^{27}$, then $W^2>2^{-k}P^{(k)}(\log W)$ for all $0\leq k\leq 6$, and so in particular Claim~\ref{cl:calculus} implies that $x^2>P(\log x)$ for all $x\geq W$.
The bound  \eqref{eq:nplogn} now implies $n<W$.

Otherwise, if $n\leq N$, \eqref{eq:conglin} bounds $n^2$ by a function which is asymptotic to a  power of $\log N$.  More specifically, we obtain (once again using the bound $\log(\log N+\log 2)<\log N$)
$$n^2\leq 2592eC(\log N)^5g(\log N),$$
where
$$g(x)=\frac{(x+\log 2+\frac{1}{2e})(x+\log2+\frac{1}{3})^3(x+\frac{2}{9}\log 2)(x+\log 2)}{x^6}.$$
It is clear that $g(x)\rightarrow 1$ as $x\rightarrow\infty$, but in fact $g(\log N)\leq 3$ for $N\geq 56$.  This gives $n\leq 9.196\times 10^{23}(\log N)^{5/2}$.
\end{proof}

The final tool needed for the proof of Theorem~\ref{th:cong} is the relation between two large values of $n$ such that $nP$ is integral.  As in the general case, we must produce a lower bound on the principal value of the elliptic logarithm of $P$.

\begin{claim}\label{cl:notzero}
Suppose  that $nP$ is an integral point, and let $z$ and $nz+m\omega_N$ be the principal values of the elliptic logarithms of $P$ and $nP$ respectively.  If $m=0$, then $n=1$.
\end{claim}

\begin{proof}
The proof proceeds just as that of Lemma~\ref{lem:mnotzero} and, as usual, we may assume that $n$ is odd.  In the proof of Claim~\ref{cl:logqbound} we obtained
\begin{equation}\label{bounda}\log|x_P|<2\log n+\log N\end{equation}
(on the assumption that $nP$ is integral).
Estimating the elliptic logarithm from below, as  in the proof of Lemma~\ref{lem:mnotzero}, we have
\begin{eqnarray*}-\log|z|&=&-\log\left|\frac{1}{2}\int_{x_P}^\infty\frac{dt}{\sqrt{t^3-N^2t}}\right|\\
&\leq& \frac{3}{2}\log 2+\frac{1}{2}\log\max\{|x_P|, 2N\}.
\end{eqnarray*}
On the other hand, by Claim~\ref{cl:zsmall} we have
$$\log|nz|\leq -\frac{n^2}{8}\log N.$$
If $|x_P|\geq 2N$, then these combine to yield
$$\frac{n^2}{8}\log N\leq -\log|z|-\log n\leq \frac{1}{2}\log N+\frac{3}{2}\log 2,$$
which gives $n\leq 2$ when $N\geq 56$ (indeed, for $N\geq 6$).
If $|x_P|<2N$, the above yields
$$\frac{n^2}{8}\log N\leq \frac{1}{2}\log (2N)-\log n+\frac{3}{2}\log 2.$$
This again bounds $n\leq 2$.  In either case, $n$ cannot be 2, and so $n=1$.
\end{proof}

\begin{claim}\label{cl:gap}
Suppose that $n_1P$ and $n_2P$ are integral with $2\leq n_1<n_2$.  Then
$$\log n_2\geq \frac{n_1^2}{8}\log(N)-\frac{1}{2}\log(N)+\log(\omega_1/2).$$
\end{claim}

\begin{proof}
We proceed as in the proof of Proposition~\ref{prop:gap}.  Using the estimate \eqref{eq:congupper}, we have
$$\left|n_iz+m_i\omega_N\right|\leq N^{-n_i^2/8}$$
for $i=1, 2$, and so
\begin{equation}\label{eq:conggap}N^{-1/2}\omega_1\leq \omega_N|n_2m_1-n_1m_2|\leq n_2N^{-n_1^2/8}+n_1N^{-n_2^2/8}\end{equation}
As in the proof of Proposition~\ref{prop:gap}, it is imperative that $n_2m_1-n_1m_2\neq 0$.  By Claim~\ref{cl:notzero}, we cannot have $m_1=0$.  On the other hand, Claim~\ref{cl:primes} ensures that $n_1$ and $n_2$ are prime, and so $n_2m_1=n_1m_2$ would imply either $n_1=n_2$, or
$|n_1|\leq m_1$, the latter contradicting the inequality $2m_1\leq n_1+1$ (which follows just as in the proof of Proposition~\ref{prop:gap}).

Returning to \eqref{eq:conggap}, one of the summands on the right must be at least  the average of the two.  If $$N^{-1/2}\frac{\omega_1}{2}\leq n_1N^{-n^2_2/8}$$ then, as $n_1<n_2$, we obtain $n_2\leq 3$, which is impossible.  Otherwise, 
$$N^{-1/2}\frac{\omega_1}{2}\leq n_2N^{-n_1^2/8}.$$
The bound above follows by taking logarithms.
\end{proof}

We are now in a position to complete the proof of Theorem~\ref{th:cong} in the case $N\geq 56$.
Suppose that $n_1P$ and $n_2P$ are both integral, with $n_1<n_2$.  Then we have
$$n_2\leq\max\{3.6\times 10^{27}, 9.196\times 10^{23}(\log N)^{5/2}\}.$$
If $n_2\leq 3.6\times 10^{27}$, recalling that we must have $n_1\geq 11$, then Claim~\ref{cl:gap} becomes
$$27\log 10+\log 3.6\geq \frac{121}{8}\log N-\frac{1}{2}\log N+0.270,$$
whence $N\leq 75$.  On the other hand, if $n_2\leq  9.196\times10^{23}(\log N)^{5/2}$, then the same claim gives us
$$\frac{5}{2}\log\log N+23\log 10+\log 9.196\geq \frac{121}{8}\log N -\frac{1}{2}\log N+0.270.$$
With some differential calculus, we can see that this implies $N\leq 54$.  

It now remains to check the claim for curves $E_N$ with $N\leq 75$.  Below we list the integral points on $E_N$ for $N$ square free and $N\leq 75$.  The data were computed, for the most part, in Magma \cite{magma}, although the values $N=66$ and $N=73$ presented some minor difficulties.  In both cases the default routines in Magma were unable to verify the rank of $E_N(\QQ)$ exactly.  The curve $E_{66}$, however, appears as curve 69696GM2 in Cremona's elliptic curve database \cite{cremona}, and the rank of $E_{73}$ may be checked with Mwrank (a program written by Cremona, and now included in SAGE \cite{sage}).  In both cases, it turns out that $E_N/\QQ$ has rank zero.

  In the table below, torsion points and points with $y_P<0$ have not been listed, and values of  $N$ with no non-torsion integral points have been omitted.  We leave it to the reader to confirm that none of these points is a multiple of another.  One way of doing this without computing the Mordell-Weil groups of the curves is to confirm that for no $N$ are there two points $P$, $Q$ on the below list with $\h(P)\geq 121\h(Q)$.
$$
\begin{array}{|c|c|}\hline
N & \text{ integral points }\\\hline
5 & (45, 300), (-4, 6)\\
6&  (18, 72), (12, 36), (-2, 8), (-3, 9), (294, 5040)\\
7&(25, 120)\\
14&(18, 48), (112, 1176)\\
15&(25, 100), (-9, 36), (60, 450)\\
21& (147, 1764), (28, 98), (-3, 36)\\
22& (2178, 101640)\\
29&(284229, 151531380)\\
30&(150, 1800), (-20, 100), (-6, 72), (45, 225)\\
34&(578, 13872), (-2, 48), (-16, 120), (162, 2016)\\
39&(975, 30420), (-36, 90)\\
41&(-9, 120), (841, 24360) \\
46&(242, 3696)\\
65&(169, 2028), (-25, 300), (-16, 252)\\
69&(1083, 35568)\\
70&(245, 3675), (-20, 300), (126, 1176)\\\hline
\end{array}
$$
This table concludes the proof of Theorem~\ref{th:cong}.
\end{proof}

\bibliographystyle{plain}
\bibliography{intpts}

\begin{thebibliography}{10}

\bibitem{ayad}
M.~Ayad.
\newblock Points {$S$}-entiers des courbes elliptiques.
\newblock {\em Manuscripta Math.}, 76(3-4):305--324, 1992.

\bibitem{bombieri}
E.~Bombieri.
\newblock On the {T}hue-{M}ahler equation. {II}.
\newblock {\em Acta Arith.}, 67(1):69--96, 1994.

\bibitem{magma}
W.~Bosma, J.~Cannon, and C.~Playoust.
\newblock The {M}agma algebra system. {I}.~the user language.
\newblock {\em J.\ Symbolic Comput.}, 24, 1997.

\bibitem{bst}
A.~Bremner, J.~H. Silverman, and N.~Tzanakis.
\newblock Integral points in arithmetic progression on {$y\sp 2=x(x\sp 2-n\sp
  2)$}.
\newblock {\em J. Number Theory}, 80(2):187--208, 2000.

\bibitem{MR1654780}
J.~Cheon and S.~Hahn.
\newblock Explicit valuations of division polynomials of an elliptic curve.
\newblock {\em Manuscripta Math.}, 97:319--328, 1998.

\bibitem{sage}
J.~E. Cremona.
\newblock {M}wrank package in {SAGE} mathematics software, version 2.10.1,
  \url{http://www.sagemath.org/}.

\bibitem{cremona}
J.~E. Cremona.
\newblock The elliptic curve database for conductors to 130000.
\newblock In {\em Algorithmic number theory}, volume 4076 of {\em Lecture Notes
  in Computer Science}. Springer, 2006.

\bibitem{david}
S.~David.
\newblock Minorations de formes lin\'eaires de logarithmes elliptiques.
\newblock {\em M\'em. Soc. Math. France (N.S.)}, 62:iv+143, 1995.

\bibitem{erdos}
P.~Erd\H{o}s.
\newblock Arithmetical properties of polynomials.
\newblock {\em J. London Math. Soc.}, 28, 1953.

\bibitem{grosssilv}
R.~Gross and J.~H. Silverman.
\newblock {$S$}-integer points on elliptic curves.
\newblock {\em Pacific J. Math.}, 167, 1995.

\bibitem{hindrysilv}
M.~Hindry and J.~H. Silverman.
\newblock The canonical height and integral points on elliptic curves.
\newblock {\em Invent. Math.}, 93(2):419--450, 1988.

\bibitem{me_quant}
P.~Ingram.
\newblock A quantitative primitive divisor result for elliptic divisibility
  sequences.
\newblock (preprint).

\bibitem{me:eds}
P.~Ingram.
\newblock Elliptic divisibility sequences over certain curves.
\newblock {\em J. Number Theory}, 123(2):473--486, 2007.

\bibitem{meandjoe}
P.~Ingram and J.~H. Silverman.
\newblock Uniform estimates for primitive divisors in elliptic divisibility
  sequences.
\newblock (preprint), 2006.

\bibitem{silvlower}
J.~H. Silverman.
\newblock Lower bounds for the canonical height on elliptic curves.
\newblock {\em Duke Math. J.}, 48, 1981.

\bibitem{silverman}
J.~H. Silverman.
\newblock {\em The arithmetic of elliptic curves}, volume 106 of {\em Graduate
  Texts in Mathematics}.
\newblock Springer-Verlag, New York, 1986.

\bibitem{silvsieg}
J.~H. Silverman.
\newblock A quantitative version of {S}iegel's theorem: integral points on
  elliptic curves and {C}atalan curves.
\newblock {\em J. Reine Angew. Math.}, 378, 1987.

\bibitem{silvcanon}
J.~H. Silverman.
\newblock The difference between the {W}eil height and the canonical height on
  elliptic curves.
\newblock {\em Math. Comp.}, 55, 1990.

\bibitem{strtza}
R.~J. Stroeker and N.~Tzanakis.
\newblock Solving elliptic {D}iophantine equations by estimating linear forms
  in elliptic logarithms.
\newblock {\em Acta Arith.}, 67(2):177--196, 1994.

\bibitem{ward}
M.~Ward.
\newblock Memoir on elliptic divisibility sequences.
\newblock {\em Amer. J. Math.}, 70:31--74, 1948.

\end{thebibliography}

\end{document}